\documentclass[11pt, reqno]{elsarticle}%
\usepackage{lipsum}
\makeatletter
\def\ps@pprintTitle{%
 \let\@oddhead\@empty
 \let\@evenhead\@empty
 \def\@oddfoot{}%
 \let\@evenfoot\@oddfoot}
\makeatother

\usepackage{amsmath}
\usepackage{bbm}
\usepackage[dvipsnames]{xcolor}

\usepackage{array}
\usepackage{amssymb}
\usepackage{amsthm}
\theoremstyle{plain}

\usepackage{dsfont} 

\usepackage[margin=1in]{geometry}
\usepackage[bookmarks=false]{hyperref}
\usepackage{epstopdf}

\def\+{{\oplus}}
\newcommand{\p}{\mathbb{P}}

\newcommand{\pa}{p_i^{\uparrow}}
\newcommand{\pd}{p_i^{\downarrow}}

\newcommand{\new}[2]{{\color{Black}#2}}

\newcounter{comments}
\newtheorem*{theorem*}{Theorem}
\newtheorem{thm}{Theorem}[section]
\newtheorem{theorem}[thm]{Theorem}

\newtheorem{cor}[thm]{Corollary}

\newtheorem{defn}[thm]{Definition}
\newtheorem{example}[thm]{Example}
\newtheorem{remark}[thm]{Remark}

\numberwithin{equation}{section}

\begin{document}

\title{The Influence of Canalization on the Robustness of Boolean Networks}

\author[uzh,usz]{C. Kadelka\corref{cor1}\fnref{fn1}}
\ead{kadelka.claus@virology.uzh.ch}

\author[bsse]{J. Kuipers}
\ead{jack.kuipers@bsse.ethz.ch}

\author[uc,jac]{R. Laubenbacher\fnref{fn1}}
\ead{laubenbacher@uchc.edu}

\cortext[cor1]{Corresponding author}
\fntext[fn1]{Supported by NSF Grant CMMI-0908201 and US DoD Grant W911NF-14-1-0486}

\address[uzh]{Institute of Medical Virology, University of Zurich, 8006 Zurich, Switzerland}
\address[usz]{Division of Infectious Diseases and Hospital Epidemiology, University Hospital Zurich, 8091 Zurich, Switzerland}
\address[bsse]{D-BSSE, ETH Zurich, Mattenstrasse 26, 4058 Basel, Switzerland}
\address[uc]{Center for Quantitative Medicine, University of Connecticut Health Center, Farmington, CT 06030, USA}
\address[jac]{Jackson Laboratory for Genomic Medicine, Farmington, CT 06030, USA}


\begin{abstract}
Time- and state-discrete dynamical systems are frequently used to model molecular networks. This paper provides a collection of mathematical and computational tools for the study of robustness in Boolean network models. The focus is on networks governed by $k$-canalizing functions, a recently introduced class of Boolean functions that contains the well-studied class of nested canalizing functions. The activities and sensitivity of a function quantify the impact of input changes on the function output. This paper generalizes the latter concept to $c$-sensitivity and provides formulas for the activities and $c$-sensitivity of general $k$-canalizing functions as well as canalizing functions with more precisely defined structure. A popular measure for the robustness of a network, the Derrida value, can be expressed as a weighted sum of the $c$-sensitivities of the governing canalizing functions, and can also be calculated for a stochastic extension of Boolean networks. These findings provide a computationally efficient way to obtain Derrida values of Boolean networks, deterministic or stochastic, that does not involve simulation.
\end{abstract}

\maketitle

\section{Introduction}
\label{sec-intro}

The robustness of dynamic networks has long been an important topic of investigation in a wide range of contexts, using various definitions of the concept \cite{Strogatz, Boccaletti}. Due to the important role of stochasticity in the dynamic behavior of biological networks, in particular gene regulatory networks, the concept of robustness has been studied extensively in this context \cite{Kitano}. Since the introduction of Boolean and logical network models to the study of the properties of gene regulatory networks \cite{Kau4, Thomas73}, time- and state-discrete models, referred to subsequently as discrete dynamical systems, have become an increasingly popular representation of molecular networks \cite{Albert03,Li04,Davidich08}. For the most part, these consist of Boolean networks and various generalizations thereof. 
Questions regarding the robustness of molecular networks, modeled in the discrete dynamical systems framework, frequently involve the relationship between structural features of the network and its resulting dynamics. One commonly used measure of the robustness of a discrete dynamic network is the so-called Derrida value of the network, a measure of how perturbations propagate through the network~\cite{Derrida1}. 
This measure can then be related to network structure, such as the type of logical rules used to represent the regulatory mechanisms for individual network nodes; see, e.g., \cite{Balleza}. 

A frequently investigated concept related to robustness is that of \emph{canalization} in developmental biology, introduced by Waddington in the 1940s~\cite{Wad}. It was intended to account for the absence of a known mechanism that enables the genetic regulatory protocols driving embryonal development to accurately produce a specific phenotype, even in light of substantial variation in the developing organism's environment. The underlying idea is that phenotypes can be thought of as valleys in a landscape, and canalization is a ``force'' that channels the developmental trajectory accurately into a particular valley, protecting it from perturbations. 
Kauffman introduced a version of this concept to Boolean network modeling of gene regulatory networks by studying canalizing functions~\cite{Kau74}, as well as the special subclass of so-called nested canalizing functions~\cite{Kau2}.
Several authors recently extended this work by considering canalization as a property of Boolean functions~\cite{Layne12,He16}. Generalizing findings from \cite{Yua1}, they showed that every Boolean function has a unique algebraic form, characterized by three invariants: canalizing depth, dominance layer numbers and the non-canalizing core-polynomial. The canalizing depth of a function describes its degree of canalization and a $k$-canalizing function is defined to have canalizing depth of at least $k$. This extension neatly stratifies the set of all Boolean functions on $n$ variables by their canalizing depth and dominance layer numbers, including so-called nested canalizing functions ($n$-canalizing), canalizing functions ($1$-canalizing) and non-canalizing functions ($0$-canalizing).

A \emph{canalizing} function possesses at least one input variable such that, if this variable takes on a certain ``canalizing'' value, then the output value is already determined, regardless of the values of the remaining input variables. If this variable takes on another value, and there is a second variable with this same property, the function is $2$-canalizing. If $k$ variables follow this pattern, the function is \emph{$k$-canalizing}, and if all variables follow this pattern, the function is \emph{nested canalizing} (NCF). By definition, any $(k+1)$-canalizing function is also $k$-canalizing, and the \new{}{Boolean function $f(x_1,\ldots, x_n) = x_1 \wedge \cdots \wedge x_n$} is an example of an NCF.

The relationship between network stability, frequently measured using Derrida values, and the proportion of canalizing functions and their degree of canalization has received much attention in recent years. Boolean networks governed by canalizing functions are more stable than those constructed using random functions~\cite{Kau2,Kau3,Karlsson07}. In general, network stability and the degree of canalization are positively correlated, with networks governed by NCFs exhibiting the most stable dynamics~\cite{Mur,Kochi12,Jansen13,Dimitrova15}. These findings clearly motivate the study of $k$-canalizing functions in the context of understanding the regulatory logic of gene networks. 

In this paper, we introduce a closed formula for the efficient computation of the Derrida values of a network governed by $k$-canalizing functions. In detail, the paper is ordered as follows.
In Section \ref{2}, we formally introduce the computational concept of canalization, see also, e.g., \cite{Kau2,He16,Yua1}.
The activity of a variable in a function quantifies the influence of that variable on the whole function, while the average sensitivity of a function measures how sensitive a function is to a random input change and can be expressed as a weighted sum of the activities. Both quantities have been extensively studied by the Boolean modeling community, see, e.g., \cite{Layne12,Shmul04,Boppana97,Liu08}. In Section \ref{3}, we derive formulas for the expected activities of any $k$-canalizing function as well as any canalizing function with known dominance layer numbers. We also introduce the $c$-sensitivity of a function as a natural generalization of the sensitivity and show how to calculate the average $c$-sensitivity of a canalizing function from the activities of its variables. 
In Section \ref{4}, we use the normalized average $c$-sensitivities to derive a formula for the Derrida values of a Boolean network that is governed by $k$-canalizing functions. This greatly simplifies the application of this tool for robustness analyses of networks, which otherwise requires extensive simulations (difficult or infeasible for large networks). We explore the relationship between the Derrida value, the canalizing depth, the dominance layer numbers and the non-canalizing core-polynomial of a Boolean function.
In Section \ref{5}, we extend the formula for the Derrida values to the case of stochastic networks, which are an important modeling paradigm for gene regulatory networks.
Section \ref{6} concludes this paper with some remarks and possible avenues of future work.
 

\section{The Concept of Canalization}\label{2} 
In this section we review some well-known concepts and definitions, mainly from \cite{He16}, to introduce the computational concept of \emph{canalization}.
\begin{defn} \label{def2.1} 
A Boolean function $f(x_{1},\ldots, x_{n})$ is essential in the variable $x_{i}$ if there \new{}{exists} $r, s\in \{0,1\}$ and $(x_1, \ldots, x_{i-1}, x_{i+1}, \ldots, x_n) \in \{0,1\}^{n-1}$ such that
\[f(x_{1},\ldots,x_{i-1},r,x_{i+1},\ldots,x_{n})\neq f(x_{1},\ldots,x_{i-1},s,x_{i+1},\ldots,x_{n}).\]
\end{defn}

\begin{defn}\label{def2.2}
A Boolean function $f: \{0,1\}^n \rightarrow \{0,1\}$ is canalizing if there \new{}{exists} a variable $x_i$, a Boolean function $g(x_1,\ldots,x_{i-1},x_{i+1},\ldots,x_n)$ and $a,b \in \{0,1\}$ such that
$$f(x_1,\ldots,x_n) = 
 \begin{cases}
  b & \ \text{if} \  x_i = a \\
  g\not\equiv b & \ \text{if} \ x_i \neq a,
  \end{cases}
$$
in which case $x_i$ is called a canalizing variable, the input $a$ is the canalizing input, and the output value $b$ when $x_i=a$ is the corresponding canalized output.
\end{defn}

\begin{defn}\label{def2.3} 
A Boolean function $f(x_1,\ldots,x_n)$ is $k$-canalizing, where $0 \leq k \leq n$, with respect to the permutation $\sigma \in \mathcal{S}_n$, inputs $a_1,\ldots,a_k$ and outputs $b_1,\ldots,b_k$, if
\begin{equation}\label{eq:kcanalizing}f(x_{1},\ldots,x_{n})=
\left\{\begin{array}[c]{ll}
b_{1} & x_{\sigma(1)} = a_1,\\
b_{2} & x_{\sigma(1)} \neq a_1, x_{\sigma(2)} = a_2,\\
b_{3} & x_{\sigma(1)} \neq a_1, x_{\sigma(2)} \neq a_2, x_{\sigma(3)} = a_3,\\
\vdots  & \vdots\\
b_{k} & x_{\sigma(1)} \neq a_1,\ldots,x_{\sigma(k-1)}\neq a_{k-1}, x_{\sigma(k)} = a_k,\\
g\not\equiv b_k & x_{\sigma(1)} \neq a_1,\ldots,x_{\sigma(k-1)}\neq a_{k-1}, x_{\sigma(k)} \neq a_k,
\end{array}\right.\end{equation}
where $g = g(x_{\sigma(k+1)},\ldots,x_{\sigma(n)})$ is a Boolean function on $n-k$ variables. When $g$ is not canalizing, the integer $k$ is the canalizing depth of $f$ (as in \cite{Layne12}), and if $g$ is in addition not constant, it is called the core function of $f$, denoted by $f_C$. \new{}{If $g$ is constant, the variables $x_{\sigma(k+1)}, \ldots, x_{\sigma(n)}$ are not essential and $f$ does not have a core function.}
\end{defn}

\new{}{While the representation of a $k$-canalizing function as in Equation \ref{eq:kcanalizing} is generally not unique (certain canalizing variables may be reordered), its canalizing depth as well as its core function (if it exists) are independent of representation \cite{He16}.

\begin{example}
The Boolean function $f(w,x,y,z) = w \wedge \bar x \wedge (y \oplus z)$ has canalizing depth $2$ and core function $f_C = y \oplus z$.
\end{example}}

\begin{remark}\label{rem_ncf} 
If we consider the set of all Boolean functions on $n$ variables, then
\begin{enumerate}[\ (a)]
 \item The $n$-canalizing functions are precisely the well-studied nested canalizing functions.
 \item The $1$-canalizing functions are precisely the canalizing functions.
 \item Every Boolean function is $0$-canalizing.
 \item Every non-canalizing function has canalizing depth $0$. 
\new{}{ \item As in \cite{He16}, constant functions are - contrary to Kauffman's original definition \cite{Kau1} - not canalizing. This implies that (i) $g$ is uniquely determined and (ii) since $g \not\equiv b_k$, a $k$-canalizing function is essential in each of its canalizing variables.}
\end{enumerate}
\end{remark}


\section{Activities and Normalized Average Sensitivities}\label{3}
Some variables of a Boolean function have a greater influence over the output of the function than others. The activity of variable $x_i$ in the function $f(x_1,\ldots,x_n)$ is defined as 
$$\alpha^f_i = \frac 1{2^n} \sum_{\mathbf x \in \{0,1\}^n} \chi[f(\mathbf x) \neq f(\mathbf x \oplus e_i)],$$
where $\chi$ is an indicator function and $e_i$ is the $i$th unit vector. A change in a highly active variable frequently affects the function $f$, while a change in a variable with activity $0$ never alters $f$.

Another important quantity, directly related to the activity of a variable and introduced in \cite{Cook86}, measures how sensitive the output of a function is to input changes. The sensitivity of a function $f$ at a vector $\mathbf x$ is defined as the number of Hamming neighbors of $\mathbf x$ (vectors that differ from $\mathbf x$ in exactly one bit) with a different function value than $f(\mathbf x)$. That is,
$$S^f(\mathbf x) = \sum_{i=1}^n \chi[f(\mathbf x) \neq f(\mathbf x \oplus e_i)].$$
The average sensitivity $S^f$ is the expected value of $S^f(\mathbf x)$\new{under the distribution of $\mathbf x$. Under the uniform distribution,}{. Assuming a uniform distribution of $\mathbf x$,}
$$S^f = \mathbb{E}[S^f(\mathbf x)] = \frac 1{2^n} \sum_{\mathbf x \in \{0,1\}^n} \sum_{i=1}^n \chi[f(\mathbf x) \neq f(\mathbf x \oplus e_i)] = \sum_{i=1}^n \alpha^f_i.$$

In this section, we will generalize the concept of sensitivity to $c$-sensitivity and calculate the normalized average $c$-sensitivity for different classes of canalizing functions, which we will then use in the next section for the calculation of stability properties of Boolean networks. 

\begin{defn}\label{def_m_sens}
Any vector that differs in exactly $c$ bits from a given vector $\mathbf x$ is called a $c$-Hamming neighbor of $\mathbf x$. The $c$-sensitivity of $f(x_1,\ldots,x_n)$ at $\mathbf x$ is defined as the number of $c$-Hamming neighbors of $\mathbf x$ on which the function value is different from its value on $\mathbf x$. That is,
$$S^f_c(\mathbf x) = \sum_{\substack{I \subseteq \{1,\ldots,n\}\\ | I | = c}} \chi[f(\mathbf x) \neq f(\mathbf x \oplus e_I)],$$
where $e_I$ is a vector with $1$ at all indices in $I$ and $0$ everywhere else. \new{}{Note that, by definition, $S_0^f(\mathbf x) = 0$.} Assuming a uniform distribution of $\mathbf x$,
$$S^f_c = \mathbb{E}[S^f_c(\mathbf x)] = \frac 1{2^n}  \sum_{\mathbf x \in \{0,1\}^n} \sum_{\substack{I \subseteq \{1,\ldots,n\}\\ | I | = c}}\chi[f(\mathbf x) \neq f(\mathbf x \oplus e_I)]$$ 
is the average $c$-sensitivity of $f$. The range of $S^f_c$ is $[0,\binom nc]$. Let us therefore define the normalized average $c$-sensitivity of $f$ as 
$$s^f_c = \frac{S^f_c}{\binom nc} \in [0,1].$$
\end{defn}

This definition generalizes the concept of sensitivity in a natural way ($S^f_1 = S^f$) and allows the impact of a simultaneous change in more than one input of a function to be studied. We will now derive the expected activities of a $k$-canalizing function, where the expectation is taken over all $k$-canalizing functions and a uniform distribution is assumed.

\begin{theorem}\label{thm_q_k_canal}
Let $f_k$ be a $k$-canalizing function of $n$ variables. By relabeling the variables if necessary, assume that $f_k$ is $k$-canalizing in the variable order $x_1, x_2, \ldots, x_k$.
The expected activity of $x_j$ in $f_k$  is 
$$\mathbb{E}[\alpha^{f_k}_{j}] = \begin{cases} \frac 1{2^{j}} & \ \text{if} \ j<k\\ 
\frac 1{2^{k-1}} \frac{2^{2^{n-k}-1}}{2^{2^{n-k}}-1} & \ \text{if} \ j=k\\
\frac 1{2^{k}} \frac{2^{2^{n-k}-1}}{2^{2^{n-k}}-1} & \ \text{if} \ j>k.\end{cases}$$
\end{theorem}

\begin{proof}
See Appendix.
\end{proof}

The expected average $c$-sensitivity of any $k$-canalizing function is a weighted sum of the activities of its variables.
\begin{theorem}\label{thm_av_m_sens}
By relabeling the variables if needed, assume that $f_k(x_1,\ldots,x_n)$ is a $k$-canalizing function in the variable order $x_1, x_2, \ldots, x_k$. The average $c$-sensitivity of $f$ is
\[S^{f_k}_c = \sum_{j=1}^{n-c+1}\binom{n-j}{c-1} \alpha^{f_k}_{j}.\]
\end{theorem}

\begin{proof}
See Appendix.
\end{proof}

\begin{cor}\label{thm_q_ncf}
The expected activities of the variables $(x_{\sigma(1)}, x_{\sigma(2)}, \ldots, x_{\sigma(n)})$ of an NCF $f$ are
$$\mathbb{E}[\alpha^f] = \left(\frac 12, \frac 14, \ldots, \frac 1{2^{n-2}}, \frac 1{2^{n-1}},\frac 1{2^{n-1}}\right),$$
and the expected normalized average $c$-sensitivity \new{}{
$$\mathbb{E}[s^f_c] = \frac{c}{2n}{}_2F_1\Big[1,c-n;1-n;\frac 12\Big]$$
can be expressed in terms of a hypergeometric function. In particular for $c=1$, $\mathbb{E}[s^f_1] = \frac 1n$.}
\end{cor}

Theorem 4.6 of \cite{He16} shows that any Boolean function can be written in a unique standard monomial form, in which the variables are grouped into different layers based on their dominance (see also \cite{Abd2,Yua1} for earlier work on this topic for NCFs). Any canalizing variable is part of the first layer. All variables that become canalizing once the variables from the first layer are excluded, are part of the second layer, etc. The number of layers is called the \emph{layer number}, denoted by $r$. The number of variables in the $i$th layer is the \emph{dominance number of layer $i$}, denoted by $k_i$, and the number of all variables that become eventually canalizing is the \emph{canalizing depth} $k=\sum k_i$. All remaining variables that never become canalizing are part of the \emph{core polynomial}, which is simply an affine transformation of the core function \new{}{(see \cite{He16} for details).}

Theorem \ref{thm_q_k_canal} yields the expected activities of a function, for which only its minimal canalizing depth is known. If the exact canalizing depth of a function and its dominance layer numbers are known, we can quantify the dynamical properties much more accurately. In \cite{Yua1}, the authors have already computed the activities and average sensitivity of an NCF with known dominance layer numbers. We will now determine the activities of a Boolean function, for which the exact canalizing depth, the dominance layer numbers and the Hamming weight of its core function are known. In this case, we do not require an expected value since all such functions share the same activities.

\begin{theorem}\label{thm_derrida3}
Let $f(x_1,\ldots,x_n)$ be any Boolean function with canalizing depth $k\geq 0$, $r$ layers of canalization with layer structure $k_1,\ldots,k_r$, where $k_i \geq 1 , \sum_{i=1}^r k_i = k$, and $v\in \{1,\ldots,2^{n-k}\}$ entries $\neq b_k$ in the truth table of its non-canalizing core $g$. By relabeling the variables if necessary, assume that $f$ is $k$-canalizing in the variable order $x_1, x_2, \ldots, x_k$. The activity of $x_j$ on $f$ is 
\begin{align*}
\alpha^f_j &= \begin{cases}
\varphi_{L(j)} + \frac{1}{2^{k-1}} \psi_{L(j)} & \ \text{if} \ j\leq k\\ 
 \frac{v(2^{n-k}-v)}{2^{n-1}(2^{n-k}-1)} & \ \text{if} \ j>k,\\ 
 \end{cases}
\end{align*}
where $L(j) \in \{1,\ldots,r\}$ denotes the dominance layer number of variable $x_j$ and
\begin{align*}
&\varphi_{r+1} = \varphi_r = 0, \qquad \varphi_i  = \varphi_{i+2} + \sum_{s=0}^{k_{i+1}-1} {\left(\frac 12\right)}^{k_1+\cdots+k_i+s} \ \text{for} \ i \geq 1,\\
&\psi_r = \frac{v}{2^{n-k}}, \qquad\quad\ \  \psi_i = 1-\psi_{i+1} \ \text{for} \ i \geq 1.
\end{align*} 
\end{theorem}

\begin{proof}
See Appendix.
\end{proof}

\section{Derrida Values of Networks Governed by $k$-canalizing functions}\label{4}
Gene regulatory networks need to be robust to perturbations. The so-called Derrida plot is a common technique to evaluate the robustness of a Boolean discrete dynamical system~\cite{Derrida1}. It describes how a perturbation of a given size propagates on average over time. If a small perturbation vanishes over time, the system is considered to be in the ordered regime. The network then typically has many steady states and short limit cycles. If the perturbation amplifies over time, the system is in the chaotic regime. A chaotic network typically possesses long limit cycles. Lastly, if the perturbation remains of similar size, then the system is situated in the narrow region between these regimes, often called the critical threshold. Many biological systems seem to operate at this ``edge of chaos''; they must be robust enough to withstand perturbations caused by environmental changes but also flexible enough to allow adaptation~\cite{Balleza, Nykter, Drossel08}.

In this section we formally define the concept of Derrida values in the framework of discrete dynamical systems, using an annealed approximation. Although this approximation corresponds to a system in which interactions are randomly rewired at each time step, it has been shown that its use does not significantly change the Derrida plot of random Boolean networks \cite{Fretter}. \new{}{With few exceptions \cite{Dimitrova15}, the calculation of Derrida values has thus far} required extensive Monte Carlo simulations~\cite{Kau2,Kau3}. We derive direct formulas for the Derrida values of Boolean networks. Especially for systems with many inputs, this offers a substantial improvement, since the time required to approximate the Derrida plot through simulations increases exponentially in the number of inputs.  

\begin{defn}\label{def_intro_derrida}
Let $F=(f_i)_{i=1}^N$ be a synchronous Boolean network of $N$ nodes. Let $I(f_i) \subseteq \{1,\ldots,N\}$ be the set of essential variables of $f_i$, \new{}{$n_i := |I(f_i)|$}. Moreover, let $\mathbf x,\mathbf y \in \{0,1\}^N$ be two system configurations that differ in $m$ coordinates. Lastly, let $J(f_i)=I(f_i) \cap \big\{j \mid x_j\neq y_j\big\} \in \{0, \ldots, m\}$ be the set of variables of $f_i$ where $\mathbf x$ and $\mathbf y$ differ. Then, for an initial perturbation of size $m$, the Derrida value of $F$ is defined as the average size of the perturbation after one update,
\begin{equation}\label{eq:derrida}
D(F,m) = \mathbb{E}\Big[d\big(F(\mathbf x),F(\mathbf y)\big) \ \big| \ d(\mathbf x,\mathbf y) = m\Big],
\end{equation}
where $d: \{0,1\}^N \times \{0,1\}^N \rightarrow \{0, 1, \ldots, N\}$ is the Hamming distance (the standard $\ell^1$ metric) and the expected value is taken uniformly over all pairs of configurations with distance $m$.
\end{defn}

\begin{theorem}\label{thm_derrida1}
The Derrida value of a synchronous Boolean network $F=(f_i)_{i=1}^N$ can be expressed as a weighted sum of the normalized average $c$-sensitivities of its update functions,
\[D(F,m) =  \sum_{i=1}^N\p\left(f_i(\mathbf x) \neq f_i(\mathbf y) \ \big| \ d(\mathbf x,\mathbf y) =m \right) = \sum_{i=1}^N \sum_{c=0}^m\ \p\left(\left|J(f_i)\right| = c\right) s^{f_i}_c,\]
where $\big|J(f_i)\big|$ follows a hypergeometric distribution,
\[\p\big(|J(f_i)| = c\big) = \frac{\binom mc \binom{N-m}{n_i-c}}{\binom N{n_i}} = \frac{\binom {n_i}c \binom{N-n_i}{m-c}}{\binom N{m}}.\]
\end{theorem}

\begin{proof}
Since $\mathbf x$ and $\mathbf y$ are synchronously updated, the update of each component is independent from the update of other components. This implies that the Derrida value is simply the sum of the probabilities that $f_i(\mathbf x)$ and $f_i(\mathbf y)$ differ after the update, $i=1,\ldots,N$, which equals the normalized average $c$-sensitivity of $f_i$. Conditioning with respect to $|J(f_i)|$ leads to the second equality.
$J(f_i)$ is the intersection of two sets so that its magnitude $|J(f_i)|$ follows a hypergeometric distribution.
\end{proof}

We will now use this theorem together with the results from Section \ref{3} to calculate average Derrida values for different Boolean networks.

\begin{figure}
\begin{center}
\includegraphics[width=.452\textwidth]{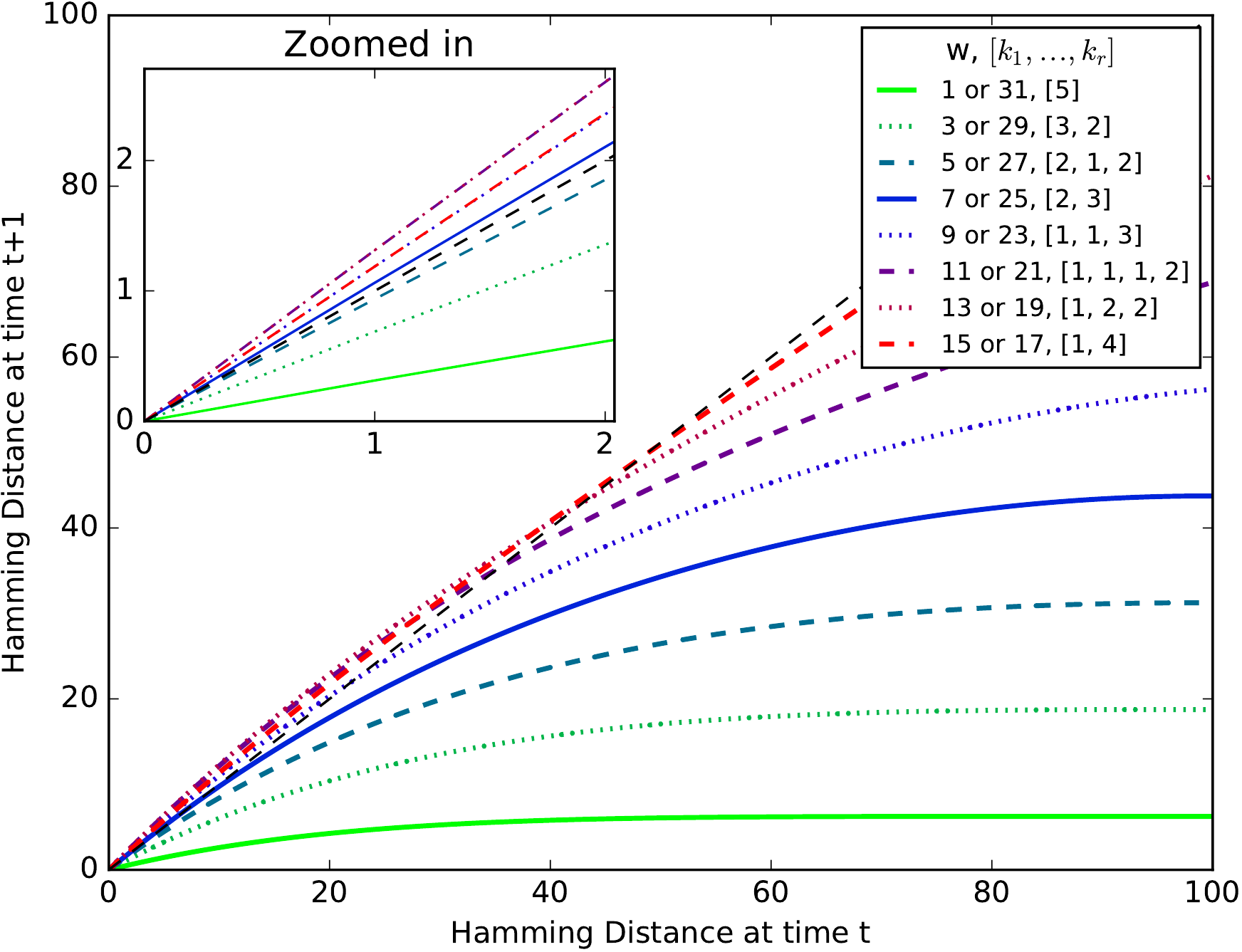} \includegraphics[width=.533\textwidth]{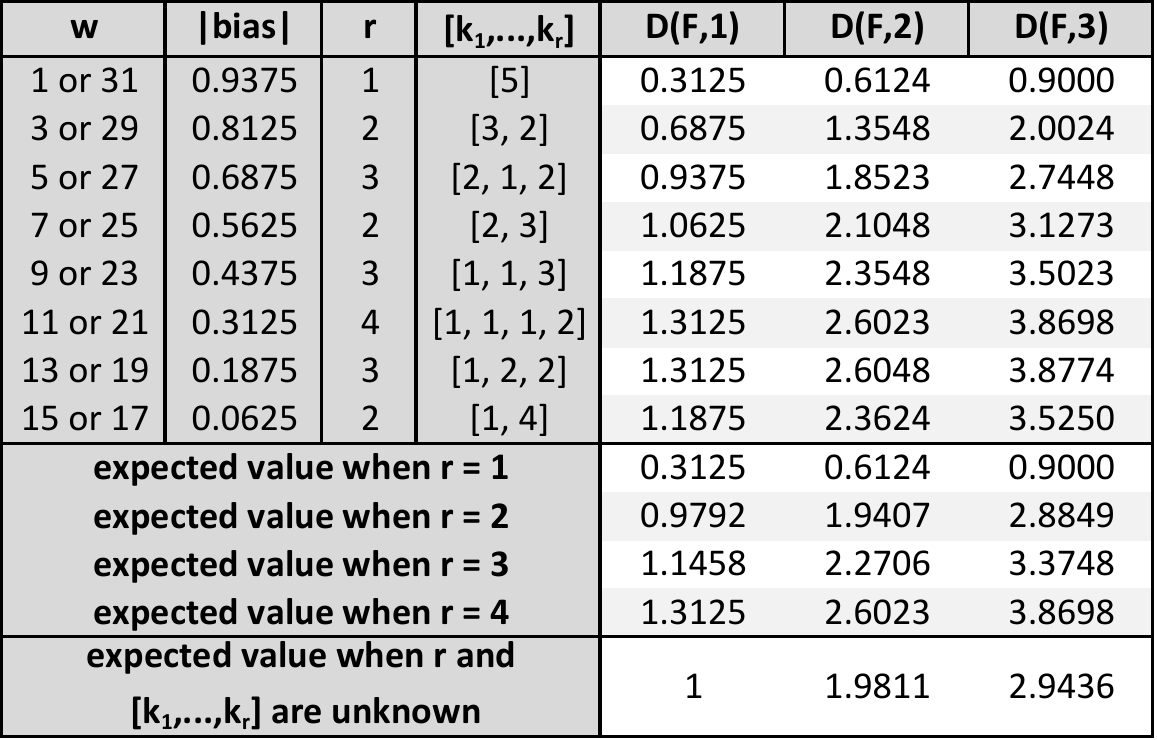}
\end{center}
\caption{Derrida plot for networks of $N=100$ genes governed by NCFs with $n=5$ inputs and varying layer structure. A black dashed line shows the line $y=x$. The table shows the Derrida values for perturbations of size up to $3$ (first 8 rows), as well as average values for cases when only the layer number but not the exact layer structure is known (rows 9-12) and average values when both are unknown (last row).}
\label{fig:derrida_plot}
\end{figure}

The Hamming weight $w$ of a Boolean function is defined as the number of 1s in its truth table, and the bias is the probability that a randomly chosen entry in the truth table is 1 minus the probability that it is 0. \new{}{The absolute bias is the absolute value of the bias.} A Boolean function with equally many 0s and 1s has bias 0 and is called balanced. Constant functions are the most biased with absolute bias 1, and it is easy to see that there is a 1-1 correspondence between the absolute bias and the layer structure of an NCF (see also \cite{Yua1}). Figure~\ref{fig:derrida_plot} depicts the Derrida values for networks of $N=100$ genes, which are governed by NCFs with $n=5$ inputs and varying layer structures (specified by the Hamming weight $w$ or, correspondingly, the number of variables in each dominance layer). The calculation of all the $800$ plotted values took less than a second on a regular desktop computer. In networks governed by highly unbalanced NCFs with Hamming weights 1, 3, 29 and 31, small perturbations vanish on average over time; these networks operate in the stable regime. Networks of NCFs with Hamming weights 5, 7, 25 and 27 operate close to the critical threshold, while networks of NCFs with Hamming weights between 11 and 23 operate in the chaotic regime. Surprisingly, networks of NCFs with Hamming weights 11, 13, 19, and 21 are more chaotic than those governed by almost balanced NCFs with Hamming weights 15 and 17.
One possible explanation for this observation may be the layer number $r$. NCFs with Hamming weight 15 or 17 consist of two layers, while NCFs with Hamming weight 13 and 19 (11 and 21) have three (four) different layers. The number of layers is positively correlated with the Derrida value for small perturbations (see rows 9-12 in Figure \ref{fig:derrida_plot}B and Table \ref{tab:spearman}). Similarly, the number of variables in the most dominant layer, $k_1$, and the absolute bias are negatively correlated. Interestingly, the correlation of the Derrida value for small perturbations with $k_1$ and with the absolute bias remains high for NCFs with many inputs, whereas the correlation with the layer number decreases with increasing number of inputs (Table \ref{tab:spearman}).

\begin{table}
\begin{center}
\includegraphics[width=0.49\textwidth]{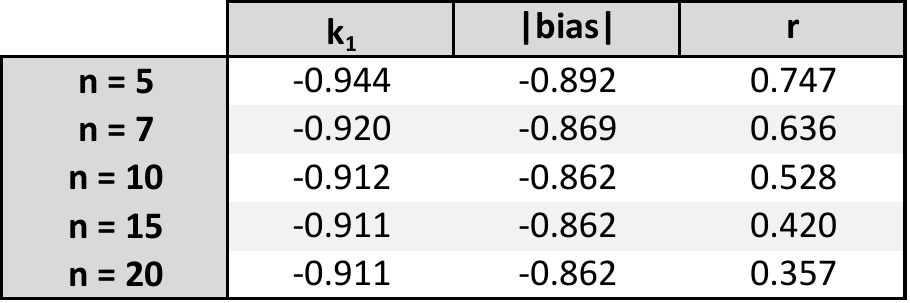}
\end{center}
\caption{For a network governed by NCFs with $n$ inputs, this table shows the Spearman correlations between the Derrida value of a single perturbation $D(F,1)$, and the number of most dominant variables $k_1$ (first column), the absolute bias (second column), and the layer number $r$ (last column).}
\label{tab:spearman}
\end{table}

Table \ref{tab:impact_single_pert_boolean} shows the impact of a single perturbation on networks governed by canalizing functions with $n=7$ inputs and canalizing depth $k=4$ for \new{}{varying layer numbers ($r$), layer structures ($k_1,\ldots, k_r$) and} numbers of entries $\neq b_k$ in the truth table of the core function ($v$). As for NCFs, the Derrida value increases with increasing layer number and decreases with increasing number of most dominant variables, $k_1$. Each combination of layer structure and $v$ yields a canalizing function with a different absolute bias and Figure \ref{fig:bias} exhibits the connection between Derrida value and absolute bias. Again, almost balanced functions give rise to more robust networks than functions with intermediate absolute bias, while networks governed by highly biased functions operate in the stable regime. Moreover, networks with a higher proportion of canalizing variables are more robust. The gain of additional dynamical stability decreases however quickly when adding canalizing variables (see \cite{Jansen13} for similar findings).

\begin{table}
\begin{center}
\includegraphics[width=0.7\textwidth]{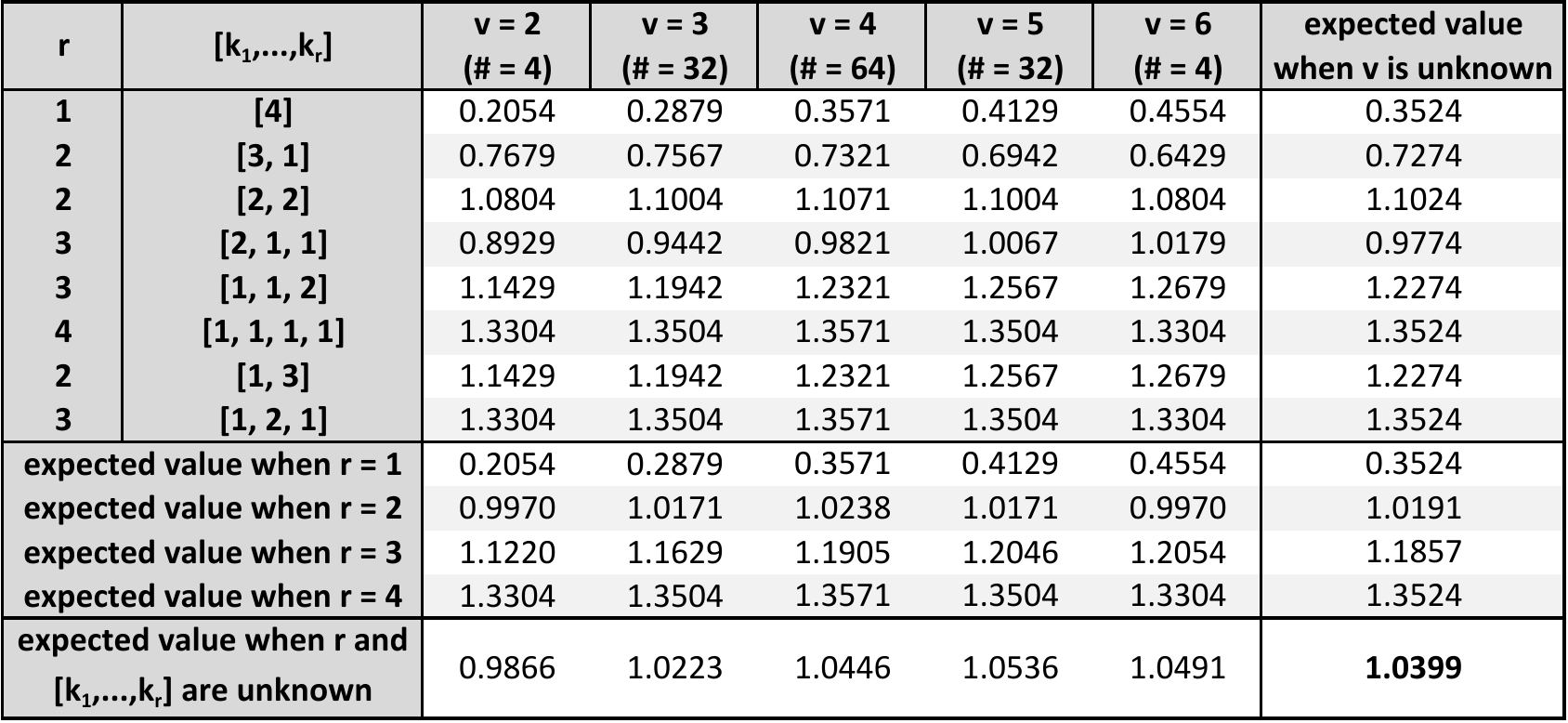}
\end{center}
\caption{Impact of a single perturbation ($D(F,1)$) on networks governed by canalizing functions with $n=7$ inputs and canalizing depth $k=4$ for \new{}{varying layer numbers ($r$), layer structures ($k_1,\ldots, k_r$) and} numbers of entries $\neq b_k$ in the truth table of the core function ($v$).}
\label{tab:impact_single_pert_boolean}
\end{table}

\begin{figure}
\begin{center}
\includegraphics[width=\textwidth]{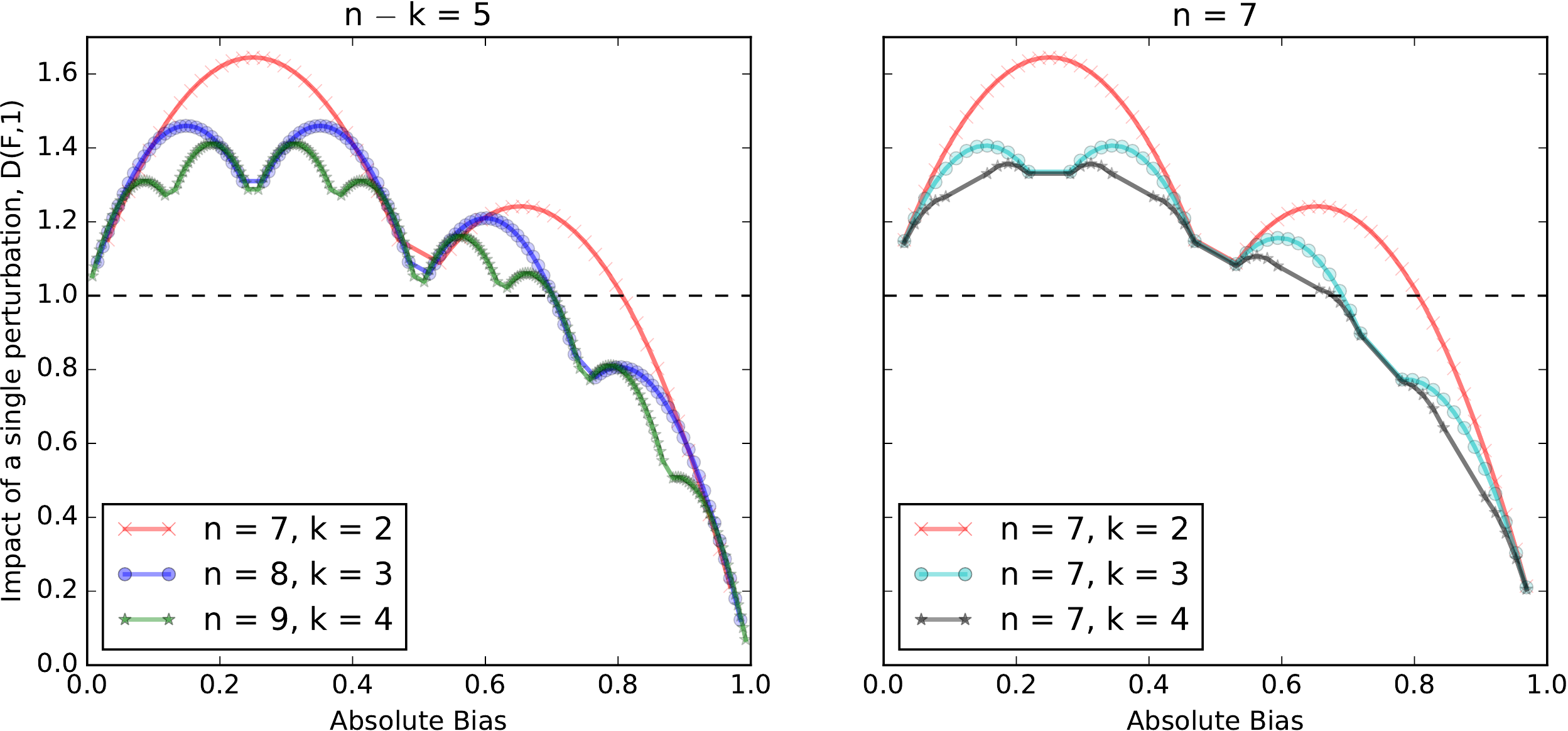}
\end{center}
\caption{For all different functions of $n$ variables with canalizing depth $k$, the impact of a single perturbation ($D(F,1)$) is plotted against the absolute bias of the function. In the left plot, the number of non-canalizing variables is constant ($n-k = 5$), while in the right plot, the total number of variables is constant ($n=7$). For visualisation purposes, the scatter points are connected by a line.}
\label{fig:bias}
\end{figure}

Theorem \ref{thm_q_k_canal} and Theorem \ref{thm_derrida3} are used in Table \ref{tab:comp_k_canal} to investigate the difference between $k$-canalizing functions (i.e., functions with canalizing depth $\geq k$) and functions with exact canalizing depth $k$. As expected, $k$-canalizing functions give rise to slightly more stable networks. \new{When the number of non-canalizing variables increases, the difference in robustness vanishes, however, quickly since the vast majority of $k$-canalizing functions has indeed also exact canalizing depth $k$ when $n \gg k$ \cite{He16}.}{This difference in robustness is however only of noticeable size for functions with few non-canalizing variables; if $k \ll n$, the vast majority of $k$-canalizing functions have also exact canalizing depth $k$ \cite{He16}.} To our knowledge, the number of non-canalizing functions with a given Hamming weight is unknown. For $n-k>4$, we therefore approximated the distribution by generating $10^7$ random non-canalizing functions. For $n-k\leq 4$, we used exhaustive enumeration.

\begin{table}
\begin{center}
\includegraphics[width=\textwidth]{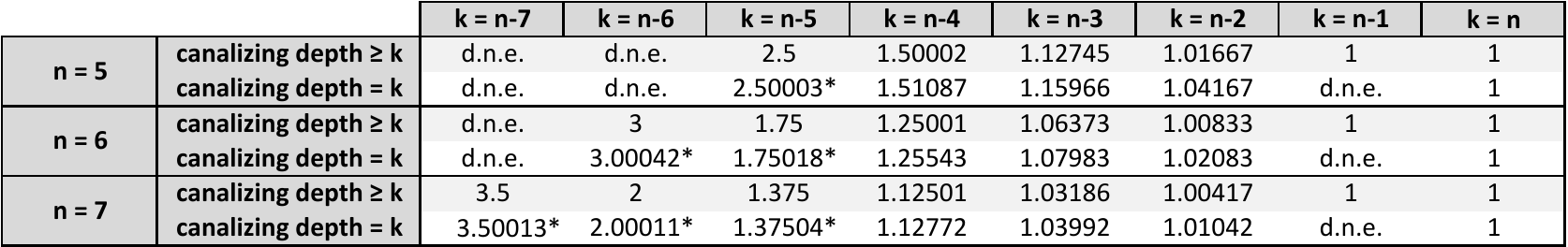}
\end{center}
\caption{Impact of a single perturbation ($D(F,1)$) on networks governed by $k$-canalizing functions (gray shaded rows) or functions with exact canalizing depth $k$ (white rows) for varying numbers of inputs ($n$) and varying values of $k$.  Approximated values are marked with *.}
\label{tab:comp_k_canal}
\end{table}

In all these analyses, we did not specify the size of the network since we focused on the impact of a single perturbation, for which it does not matter. When considering larger perturbations, the network size has a theoretical impact on the Derrida value. \new{}{By comparing the computed values for different network sizes, we found that the empirical impact is, however, negligible} as long as the proportion of perturbed nodes remains small.

\section{Derrida Values of Stochastic Discrete Dynamical Systems}\label{5}
Gene regulatory networks are stochastic in nature. A recently introduced generalization of Boolean networks, called Stochastic Discrete Dynamical Systems (SDDS), captures this inherent stochasticity by assigning gene-specific activation 
probabilities $p_i^\uparrow$ and degradation probabilities $p_i^\downarrow$, which describe how likely a specific state change happens at a given update step~\cite{SDDS}. 
This framework allows modeling of different time scales as well as stochastic variability, 
while preserving the simplicity of a Boolean network model. The Derrida value concept (Equation \ref{eq:derrida}) can be directly applied to this type of system, and we derive a formula for its computation.

\new{}{\begin{defn} \cite{SDDS}
An SDDS in the variables $x_1, \ldots, x_n$ is defined as a collection of $n$ triplets 
$$F = \left(f_i,p_i^\uparrow,p_i^\downarrow\right)_{i=1}^N,$$
where $f_i : \{0,1\}^n \rightarrow \{0,1\}$ is the update function, $p_i^\uparrow$ the activation probability, and $p_i^\downarrow$ the degradation probability for $x_i$ for all $i = 1, \ldots, n$.

For all $i$ such that $f_i(\mathbf x) \neq x_i$,
\begin{align*}
\mathbb{P}(x_i \rightarrow x_i) &= \begin{cases} 1-p_i^\uparrow & \ \text{if}\ x_i < f_i(x),\\
1-p_i^\downarrow & \ \text{if}\ x_i > f_i(x).\end{cases}\\
\mathbb{P}(x_i \rightarrow f_i(\mathbf x)) &= \begin{cases} p_i^\uparrow & \qquad \text{if}\ x_i < f_i(x),\\
p_i^\downarrow & \qquad \text{if}\ x_i > f_i(x).\end{cases}\\
\end{align*}
\end{defn}}

\begin{theorem}\label{thm_derrida4}
Let $F=\left(f_i,p_i^\uparrow,p_i^\downarrow\right)_{i=1}^N$ be an SDDS. \new{}{As in Defintion \ref{def_intro_derrida}, let $I(f_i)$ be the set of inputs of $f_i$ with $n_i := |I(f_i)|$. Further, let
$$H_{N,m,n}(c) = \frac{\binom mc \binom{N-m}{n-c}}{\binom N{n}}$$
denote the hypergeometric probability mass function.
The Derrida value of F is
\begin{align*}
D(F,m) = \sum_{i=1}^N \sum_{c=0}^m\ 
\begin{cases}
\frac mN H_{N-1,m-1,n_i-1}(c) \left[\gamma_1s^{f_i}_{c+1} + \gamma_2\left(1-s^{f_i}_{c+1}\right)\right] +  & \\
\quad + \frac{N-m}N H_{N-1,m,n_i-1}(c) \left[\gamma_3s^{f_i}_c + \gamma_4\left(1-s^{f_i}_c\right)\right] & \ \text{if}\ i \in |I(f_i)|\\
& \\
\frac mN H_{N-1,m-1,n_i}(c) \left[\gamma_1s^{f_i}_c + \gamma_2\left(1-s^{f_i}_c\right)\right] +  &\\
\quad + \frac{N-m}N H_{N-1,m,n_i}(c) \left[\gamma_3s^{f_i}_c + \gamma_4\left(1-s^{f_i}_c\right)\right] & \ \text{if}\ i \not\in |I(f_i)|\\
\end{cases}
\end{align*}
where 
\begin{align*}
\gamma_1 &= 1-\frac12 \left(\pa + \pd\right) + \pa \pd,\\
\gamma_2 &= 1 - \frac 12 \left(\pa + \pd\right),\\
\gamma_3 &= \frac 12 \left(\pa + \pd\right),\\
\gamma_4 &= \frac 12 \left(\pa + \pd\right) -\frac 12 \left(\pa\pa + \pd\pd\right).
\end{align*}}
\end{theorem}

\begin{proof}
See Appendix.
\end{proof}


\section{Discussion}\label{6}
A characteristic ability of organisms is their capability to operate in highly variable environments, striking a fine balance between the ability to adapt to changing conditions and the robustness to function predictably. This extends to the molecular networks that drive everything from  embryonal development to metabolism. Understanding the mechanisms that confer this ability is one of the central challenges in studying biology and its fundamental principles. As in physics, mathematical modeling and analysis is an enabling technology for the drive to connect structural properties of dynamic biological networks to their dynamics. The study of robustness, as instantiated in the concept of canalization, is no exception, and many such studies have been published, see, e.g., \cite{Balleza,Kau3,Fretter,Rocha}. 

Here, we have focused on one computational instantiation of robustness, that of canalization in the context of Boolean discrete dynamical systems. This includes, as a special case, nested canalization. We have provided a collection of practical and theoretical tools for the analysis of systems governed by $k$-canalizing functions. In order to study the impact of a simultaneous change in more than one input of a function, we have, at the function level, generalized the concept of sensitivity. At the network level, this allowed us to provide  easy-to-use closed form formulas for the Derrida values, a commonly used metric for the stability of networks. We explored the relationship between Derrida value, the canalizing depth, absolute bias, number of layers and number of most dominant variables of a function. In addition, we derived formulas for the Derrida values of a stochastic discrete dynamical system, a modeling framework that can cope with the inherent stochasticity of gene regulatory networks.

The presented formulas significantly simplify the study of robustness via Derrida values. While we started to disentangle the influence of the different parameters of a canalizing function on its robustness, much work remains to be done. Knowledge of the exact number of non-canalizing functions with a certain number of variables and a certain Hamming weight would be helpful in this effort.

\bibliographystyle{elsarticle-num}
\bibliography{ncf_bib_latest}

\section*{Appendix}
\subsubsection*{Theorem \ref{thm_q_k_canal}}
\begin{proof}
Let $f_k(x_1,\ldots,x_n)$ be a $k$-canalizing function with canalizing order $x_1,x_2, \ldots,x_k$, inputs $a_i$ and outputs $b_i$, $1\leq i \leq k$. 
We will use a similar argument as in \cite{Layne12} to find the expected activities of $f_k$. By definition, the activity of $x_j$ in $f_k$ is the probability that a change in $x_j$ changes the output of $f_k$. If $x_j$ is a canalizing variable (i.e., $j \leq k$), a change in $x_j$ can only affect the output of $f_k$ if none of the variables $x_1,\ldots, x_{j-1}$ receive their canalizing input. Thus,
\begin{align*}
\alpha^{f_k}_j &= \p\big({f_k}(\mathbf x) \neq {f_k}(\mathbf x \oplus e_j)\big)\\
&= \p(x_1 \neq a_1, \ldots, x_{j-1} \neq a_{j-1})\p\big({f_k}(\mathbf x) \neq {f_k}(\mathbf x \oplus e_j)\mid x_1 \neq a_1, \ldots, x_{j-1} \neq a_{j-1}\big).
\end{align*} 

Since each canalizing variable receives its canalizing input with probability $\frac 12$,
$$\p(x_1 \neq a_1, \ldots, x_{j-1} \neq a_{j-1}) = \frac{1}{2^{j-1}}$$

For any $j\leq k$, the subfunction ${f_k}(1-a_1,\ldots,1-a_{j-1},x_j,x_{j+1},\ldots,x_n)$ is canalizing in $x_j$ and can therefore be written as $(x_j+a_j)\bar g(x_{j+1},\ldots,x_n) + b_j$ for some Boolean polynomial $\bar g\not\equiv 0$ as in \cite{He16}. Hence,

\begin{align*}
\p\big({f_k}(\mathbf x) \neq {f_k}(\mathbf x \oplus e_j)\mid x_1 \neq a_1, \ldots, x_{j-1} \neq a_{j-1}\big) &= \p\big(\bar g(x_{j+1},\ldots,x_n) = 1\big).
\end{align*}

If $j<k$, $\bar g\not\equiv \mathbf 1$ since a $k$-canalizing function must be essential in all its canalizing variables (Remark \ref{rem_ncf}e). Both constant functions are thus excluded from the set of possible choices for $\bar g$ so that 
$$\p\big(\bar g(x_{j+1},\ldots,x_n) = 1\big) = \frac 12.$$

If $j=k$, $\bar g \equiv \mathbf 1$ does not cause a contradiction. In this case, there are $2^{2^{n-k}}$ choices of Boolean functions for $\bar g$, half of which satisfy $\bar g(x_{k+1},\ldots,x_n) = 1$. Only $\mathbf 0$ is excluded from the set of choices for $\bar g$ so that
\begin{equation}\label{eq_almost_half1}
\p\big(\bar g(x_{k+1},\ldots,x_n) = 1\big) = \frac{\frac 12 2^{2^{n-k}}}{2^{2^{n-k}}-1} = \frac{2^{2^{n-k}-1}}{2^{2^{n-k}}-1}
\end{equation}

On the other hand, if $x_j$ is a non-canalizing variable ($j>k$), then a change in $x_j$ can only affect the output of ${f_k}$ if none of the $k$ canalizing variables receive their canalizing input. Thus,
\begin{align*}
\alpha^{f_k}_j &= \p\big({f_k}(\mathbf x) \neq {f_k}(\mathbf x \oplus e_j)\big)\\
&= \p(x_1 \neq a_1, \ldots, x_{k} \neq a_{k})\p\big({f_k}(\mathbf x) \neq {f_k}(\mathbf x \oplus e_j)\mid x_1 \neq a_1, \ldots, x_{k} \neq a_{k}\big)\\
&= \frac 1{2^{k}}\p\big(g(x_{k+1},\ldots,x_n) \neq g(y_{k+1},\ldots,y_n)\big),
\end{align*}
where $g\not\equiv b_k$ from Definition \ref{def2.3} and $\mathbf y = \mathbf x \oplus e_j$. Similar arguments as for Equation \ref{eq_almost_half1} yield
\begin{equation}\label{eq_almost_half2}
\p\big(g(x_{k+1},\ldots,x_n) \neq g(y_{k+1},\ldots,y_n)\big) = \frac{\frac 12 2^{2^{n-k}}}{2^{2^{n-k}}-1} = \frac{2^{2^{n-k}-1}}{2^{2^{n-k}}-1}.
\end{equation}
\end{proof}

\subsubsection*{Theorem \ref{thm_av_m_sens}}
\begin{proof}
Let $\mathbf x \in \{0,1\}^n$ be a randomly chosen vector and let $\mathbf y = \mathbf x \oplus e_I$ be its $c$-Hamming neighbor, $d(\mathbf x,\mathbf y) = c$. \new{}{Let ${f_k}$ be a $k$-canalizing function in the variable order $x_1, x_2, \ldots, x_k$. Since the expected activity of all non-canalizing variables is equal (Theorem \ref{thm_q_k_canal}), the probability that ${f_k}(\mathbf x)$ and ${f_k}(\mathbf y)$ differ only depends on the first variable ($\min(I)$) where $\mathbf x$ and $\mathbf y$ differ,}
 $$ \frac 1{2^n}  \sum_{\mathbf x \in \{0,1\}^n} \chi\left[{f_k}\left(\mathbf x\right) \neq {f_k}\left(\mathbf x \oplus e_I\right)\right] = \p\left({f_k}(\mathbf x) \neq {f_k}(\mathbf x \oplus e_I)\right) = \p\left({f_k}(\mathbf x) \neq {f_k}(\mathbf x \oplus e_{\min(I)})\right) = \alpha^{f_k}_{\min(I)}.$$

There are $\binom nc$ $c$-subsets in the set $\{1,\ldots,n\}$, $\binom {n-j}{c-1}$ of which contain $j$ as its lowest element, $j=1,\ldots,n-c+1$.
Therefore, 
\begin{align*}
S^{f_k}_c &=\sum_{\substack{I \subseteq \{1,2,\ldots,n\}\\ | I | = c}}  \frac 1{2^n}  \sum_{\mathbf x \in \{0,1\}^n} \chi[{f_k}(\mathbf x) \neq {f_k}(\mathbf x \oplus e_I)] =\sum_{\substack{I \subseteq \{1,2,\ldots,n\}\\ | I | = c}} \alpha^{f_k}_{\min(I)} = \sum_{j=1}^{n-c+1}\binom{n-j}{c-1} \alpha^{f_k}_{j}.
\end{align*}
\end{proof}

\subsubsection*{Theorem \ref{thm_derrida3}}
\begin{proof}
(i) As in the proof of Theorem \ref{thm_q_k_canal}, the activity of any canalizing variable $x_j$ ($j\leq k$) is
\begin{align*}
\alpha^f_j &= \p(x_1 \neq a_1, \ldots, x_{j-1} \neq a_{j-1})\p\big(f(\mathbf x) \neq f(\mathbf x \oplus e_j)\mid x_1 \neq a_1, \ldots, x_{j-1} \neq a_{j-1}\big)\\
&= \frac 1{2^{j-1}}\p\big(f(1-a_1,\ldots,1-a_j,x_{j+1},\ldots,x_n) \neq b_j\big)
\end{align*} 

Due to the canalizing nature of $f$, the probability can be further written as

\begin{align*}
& \p\big(f(1-a_1,\ldots,1-a_j,x_{j+1},\ldots,x_n) \neq b_j\big)\\
= &\sum_{i=j+1}^k \Big[ \p\big( x_{j+1} \neq a_{j+1}, \ldots, x_{i-1} \neq a_{i-1}, x_i = a_i \big)\\
& \ \ \ \qquad \qquad \cdot \p\big( f(1-a_1,\ldots,1-a_{i-1}, a_i, x_{i+1}, \ldots, x_n) \neq b_j \mid x_{j+1} \neq a_{j+1}, \ldots, x_{i-1} \neq a_{i-1}, x_i = a_i  \big) \Big]\\
&\ \ + \p\big( x_{j+1} \neq a_{j+1}, \ldots, x_k \neq a_k \big) \p\big( f(1-a_{1},\ldots,1-a_k, x_{k+1}, \ldots, x_n) \neq b_j \mid x_{j+1} \neq a_{j+1}, \ldots, x_k \neq a_k \big) \\
= &\sum_{i=j+1}^k \frac{1}{2^{i-j}} \chi(b_i \neq b_j) + \frac{1}{2^{k-j}} \p\big( g(x_{k+1},\ldots,x_n) \neq b_j \big).
\end{align*}

Let $L := L(j)$ and $L(i)$ be the layer of the variables $x_j$ and $x_i$, $i\leq k$. The canalized output is equal for all variables of the same layer and alternates among layers. Therefore, $b_i \neq b_j$ if and only if $L(i)-L(j)$ is odd, and the first sum can be rewritten as a sum over all $k_{L+2t-1}$ variables of every second layer $L+2t-1>L$ ($t\geq 1$),
\begin{align*}
\sum_{i=j+1}^k \frac{1}{2^{i-j}} \chi(b_i \neq b_j) &= \sum_{t=1}^{\lceil (r-L)/2 \rceil}  \sum_{s=1}^{k_{L+2t-1}} {\left(\frac{1}{2}\right)}^{k_1+\ldots+k_{L+2t-2}+s-j} \\
 \end{align*}

Also, since $g$ contains $v$ entries $\neq b_k$ in its truth table,
  \begin{align*}
  \psi_l := \p\left( g(x_{k+1},\ldots,x_n) \neq b_j \right) &= \begin{cases}
   \p\left( g(x_{k+1}, \ldots, x_n) \neq b_k \right) & \ \text{if $r-l$ is even}\\
   \p\left( g(x_{k+1},\ldots,x_n) = b_k \right) & \ \text{if $r-l$ is odd}
  \end{cases}\\
  &= \begin{cases}
   \frac{v}{2^{n-k}} & \ \text{if $r-l$ is even}\\
  1-\frac{v}{2^{n-k}} & \ \text{if $r-l$ is odd}
  \end{cases}
  \end{align*}

Altogether,
\begin{align*}
 \alpha^f_j  &= \frac{1}{2^{j-1}} \left[ \sum_{t=1}^{\lceil (r-L)/2 \rceil}  \sum_{s=1}^{k_{L+2t-1}} {\left(\frac{1}{2}\right)}^{k_1+\ldots+k_{L+2t-2}+s-j} + \frac{1}{2^{k-j}} \psi_l \right]\\
&= \underbrace{   \sum_{t=1}^{\lceil (r-L)/2 \rceil}  \sum_{s=0}^{k_{L+2t-1}-1} {\left(\frac{1}{2}\right)}^{k_1+\ldots+k_{L+2t-2}+s} }_{:=\varphi_l}   + \frac{1}{2^{k-1}} \psi_l,
\end{align*} 
\new{}{By induction, $\varphi_l$ and $\psi_l$ can be calculated recursively with
\begin{align*}
&\varphi_{r+1} = \varphi_r = 0, \qquad \varphi_i  = \varphi_{i+2} + \sum_{s=0}^{k_{i+1}-1} {\left(\frac 12\right)}^{k_1+\ldots+k_i+s} \ \text{for} \ i \geq 1,\\
&\psi_r = \frac{v}{2^{n-k}}, \qquad\quad\ \  \psi_i = 1-\psi_{i+1} \ \text{for} \ i \geq 1.
\end{align*}}

(ii) On the other hand, the activity of any non-canalizing variable $x_j$ ($j>k$) is simply
\begin{align*}
\alpha^f_j &= \p(x_1 \neq a_1, \ldots, x_{k} \neq a_{k})\p\big(f(\mathbf x) \neq f(\mathbf x \oplus e_j)\mid x_1 \neq a_1, \ldots, x_{k} \neq a_{k}\big)\\
&= \frac 1{2^{k}}\p\big(g(x_{k+1},\ldots,x_n) \neq g(x_{k+1},\ldots,x_{j-1},1-x_{j},x_{j+1},\ldots,x_n)\big)\\
&= \frac 1{2^{k}}\frac{v(2^{n-k}-v)}{\binom{2^{n-k}}{2}}\\
&= \frac{v(2^{n-k}-v)}{2^{n-1}(2^{n-k}-1)}
\end{align*} 
\end{proof}

\subsubsection*{Theorem \ref{thm_derrida4}}
\begin{proof}
Let $\mathbf x, \mathbf y \in \{0,1\}^N$ be two randomly chosen vectors that differ in $m$ positions. For each node $i \in \{1,\ldots,N\}$, define three events 
\begin{align*}
A_i&=\left\{x_i \neq y_i\right\}, \\
B_i&= \left\{f_i(\mathbf x)\neq f_i(\mathbf y) \ \text{before applying}\ \pa,\pd\right\},\\
C_i&= \left\{f_i(\mathbf x)\neq f_i(\mathbf y) \ \text{after applying}\ \pa,\pd\right\}.
\end{align*}

\begin{itemize}
\item Since $\mathbf x$ and $\mathbf y$ differ in $m$ out of $N$ positions, $\p(A_i) = \frac mN$. 
\item $\p(B_i | c) := \p\big(B_i\big| |J(f_i)| = c\big)$ is simply the normalized average $c-$sensitivity of $f$, i.e., $\p(B_i | c) = s^{f_i}_c$.
\item Lastly, the probability that $\mathbf x$ and $\mathbf y$ differ after applying the propensity probabilities needs to be calculated. This probability obviously depends on $A_i$ and $B_i$. If $x_i\neq y_i$ and $f_i(\mathbf x)\neq f_i(\mathbf y)$ before applying $\pa,\pd$, we can assume that $x_i=0,y_i=1$. Then, either $f_i(\mathbf x)=0,f_i(\mathbf y)=1$ or $f_i(\mathbf x)=1,f_i(\mathbf y)=0$, both with probability $\frac 12$. In the first case, there is no change in values so that the propensity probabilities $\pa,\pd$ play no role and $f_i(\mathbf x)\neq f_i(\mathbf y)$ after applying $\pa,\pd$ with probability $1$. In the second case, $f_i(\mathbf x)$ and $f_i(\mathbf y)$ only differ after applying $\pa,\pd$, if either both updates happen (probability $\pa\pd$) or neither update happens (probability $(1-\pa)(1-\pd)$). That means,
\[\gamma_1 := \p\left(C_i|A_i,B_i\right) = \frac 12 \cdot 1+\frac 12\left(\pa\pd + \left(1-\pa\right)\left(1-\pd\right)\right) = 1-\frac 12\left(\pa+\pd\right)+\pa\pd .\]
Similarly, we can derive
\begin{align*}
\gamma_2 :=\p\left(C_i|A_i,\neg B_i\right) &= 1 - \frac 12 \left(\pa + \pd\right),\\
\gamma_3 :=\p\left(C_i|\neg A_i,B_i\right) &= \frac 12 \left(\pa + \pd\right),\\
\gamma_4 :=\p\left(C_i|\neg A_i,\neg B_i\right) &= \frac 12 \left(\pa + \pd\right) -\frac 12 \left(\pa\pa + \pd\pd\right).
\end{align*}
\end{itemize}
\new{}{
Since the exact update functions of $F$ are known, we need to distinguish two cases in the calculation of $\p(C_i)$.

\underline{Case 1} $i \in I(f_i)$: If node $i$ is self-regulatory, its update depends only on $n_i-1$ other inputs. Moreover, if $A_i$ is true, $\mathbf x$ and $\mathbf y$ differ in at least one of the essential variables of $f_i$, i.e., $J(f_i)\geq 1$. Due to the same arguments as in Theorem \ref{thm_derrida1}, $|J(f_i)|$ follows a hypergeometric distribution, and we have
\begin{align*}
&\p(|J(f_i)| = c+1 \mid A_i) = H_{N-1,m-1,n_i-1}(c),\\
&\p(|J(f_i)| = c \mid \neg A_i) = H_{N-1,m,n_i-1}(c).
\end{align*}
For these reasons, 
\begin{align*}
\p(C_i) &= \sum_{c=0}^{m-1}\ \p(|J(f_i)| = c+1 \mid A_i)\Big(\p(C_i|A_i,B_i)\p(B_i|c+1)\p(A_i) + \p(C_i|A_i,\neg B_i)\p(\neg B_i|c+1)\p(A_i)\Big) \\
&\ \ + \sum_{c=0}^{m}\ \p(|J(f_i)| = c \mid \neg A_i) \Big(\p(C_i|\neg A_i,B_i)\p(B_i|c)\p(\neg A_i) + \p(C_i|\neg A_i,\neg B_i)\p(\neg B_i|c)\p(\neg A_i) \Big)\\
&= \sum_{c=0}^{m-1} \frac mN H_{N-1,m-1,n_i-1}(c) \left[\gamma_1s^{f_i}_{c+1} + \gamma_2\left(1-s^{f_i}_{c+1}\right)\right] + \sum_{c=0}^{m}\frac{N-m}{N} H_{N-1,m,n_i-1}(c) \left[\gamma_3s^{f_i}_c + \gamma_4\left(1-s^{f_i}_c\right)\right] .
\end{align*}

\underline{Case 2} $i \not\in I(f_i)$: If node $i$ does not regulate its own update, the $n_i$ inputs of $f_i$ are among the remaining $N-1$ nodes. Similar to Case 1, we have
\begin{align*}
&\p(|J(f_i)| = c \mid A_i) = H_{N-1,m-1,n_i}(c),\\
&\p(|J(f_i)| = c \mid \neg A_i) = H_{N-1,m,n_i}(c),
\end{align*}
and
\begin{align*}
\p(C_i) &= \sum_{c=0}^{m}\ \p(|J(f_i)| = c \mid A_i)\Big(\p(C_i|A_i,B_i)\p(B_i|c)\p(A_i) + \p(C_i|A_i,\neg B_i)\p(\neg B_i|c)\p(A_i)\Big) \\
&\ \ + \sum_{c=0}^{m}\ \p(|J(f_i)| = c \mid \neg A_i) \Big(\p(C_i|\neg A_i,B_i)\p(B_i|c)\p(\neg A_i) + \p(C_i|\neg A_i,\neg B_i)\p(\neg B_i|c)\p(\neg A_i) \Big)\\
&= \sum_{c=0}^{m}\frac mN H_{N-1,m-1,n_i}(c) \left[\gamma_1s^{f_i}_{c} + \gamma_2\left(1-s^{f_i}_{c}\right)\right] + \sum_{c=0}^{m}\frac{N-m}{N}H_{N-1,m,n_i}(c) \left[\gamma_3s^{f_i}_c + \gamma_4\left(1-s^{f_i}_c\right)\right] .
\end{align*}

Since $H_{N,m,n}(c) = 0$ if $c \not\in \{\max(m+n-N,0), \ldots, \min(m,n)\}$, we can use the same summation limits for all cases. Thus,

\begin{align*}
D(F,m) &= \sum_{i=1}^N \p(C_i)\\
&= \sum_{i=1}^N \sum_{c=0}^m\ 
\begin{cases}
\frac mN H_{N-1,m-1,n_i-1}(c) \left[\gamma_1s^{f_i}_{c+1} + \gamma_2\left(1-s^{f_i}_{c+1}\right)\right]  & \\
\quad + \frac{N-m}N H_{N-1,m,n_i-1}(c) \left[\gamma_3s^{f_i}_c + \gamma_4\left(1-s^{f_i}_c\right)\right] & \ \text{if}\ i \in |I(f_i)|\\
& \\
\frac mN H_{N-1,m-1,n_i}(c) \left[\gamma_1s^{f_i}_c + \gamma_2\left(1-s^{f_i}_c\right)\right]  &\\
\quad + \frac{N-m}N H_{N-1,m,n_i}(c) \left[\gamma_3s^{f_i}_c + \gamma_4\left(1-s^{f_i}_c\right)\right] & \ \text{if}\ i \not\in |I(f_i)|\\
\end{cases}
\end{align*}
}
\end{proof}

\end{document}